\documentclass[a4,12 pt,twosided,reqno]{amsart}
\usepackage{amscd,amssymb}
\usepackage[arrow,matrix]{xy}
\usepackage{graphicx}
\usepackage{color}
\newfont{\black}{msbm10 scaled 1200}

\oddsidemargin=0.3in \evensidemargin=0.3in \theoremstyle{plain}
\newtheorem{theorem}[subsection]{Theorem}
\newtheorem{lemma}[subsection]{Lemma}
\newtheorem{proposition}[subsection]{Proposition}

\theoremstyle{definition}

\title{\bf {One-weight Inequalities for Riemann-Liouville and Hardy  Transform in Lebesgue Space with Variable Exponent}  }
\begin{document}
\begin{center}
{\bf  Riemann-Liouville and higher dimensional Harday operators for non-negative decreasing function  in $L^{p(\cdot)}$ spaces}
\end{center}
\vskip+.5cm
\begin{center}
{Ghulam Murtaza and Muhammad Sarwar}
\end{center}
\vskip+.5cm

{\bf Abstract.} In this paper one-weight inequalities  with general weights for  Riemann-Liouville transform and $ n-$ dimensional fractional integral operator in variable exponent Lebesgue spaces defined on $\mathbb{R}^{n}$ are investigated. In particular, we derive necessary and sufficient conditions  governing one-weight inequalities for these operators on the cone of non-negative decreasing functions in $L^{p(x)}$ spaces.

\vskip+0.2cm
{\bf Key Words and Phrases}: Variable exponent Lebesgue spaces, Riemann-Liouville transform,   $n-$ dimensional  fractional integral operator, one--weight inequality.

\section{introduction}
We derive   necessary and sufficient conditions  governing the one-weight inequality for the Riemann-Liouville operator
$$ R_{\alpha}f(x)=\frac{1}{x^{\alpha}}\int\limits_{0}^{x}\frac{f(t)}{(x-t)^{1-\alpha}}dt \ \ \ \  \ 0<\alpha<1,$$
and  $n-$dimensional fractional integral operator
$$I_{\alpha}g(x)=\frac{1}{|x|^{\alpha}}\int\limits_{|y|<|x|}\frac{g(t)}{|x-t|^{n-\alpha}}dt\ \ \ 0<\alpha<n, $$
on the cone  of non-negative decreasing  function in $L^{p(x)}$ spaces.

In the last two decades a considerable interest of researchers was attracted to the investigation of the mapping properties of integral operators in so called Nakano spaces $L^{p(\cdot)}$ (see e.g., the monographs \cite{CF}, \cite{DHHR} and references therein). Mathematical problems related to these spaces arise in applications  to mechanics of the continuum medium. For example, M. Ru\v{z}i\v{c}ka \cite{Ruz} studied
the problems in the so called rheological and electrorheological fluids, which lead to spaces
with variable exponent.

Weighted estimates for the Hardy transform
$$ (Hf)(x)= \int\limits_0^x f(t) dt, \;\;\; x>0,  $$
in  $L^{p(\cdot)}$ spaces were derived in the papers \cite{DS} for power-type weights and in \cite{EdKoMe1}, \cite{EdKoMe2}, \cite{K}, \cite{CrMa}, \cite{MaHa}  for general weights. The Hardy inequality for non-negative decreasing functions was studied  in \cite{SJ}, \cite{BS2}.

Weighted problems for the Riemann-Liouville transform in $L^{p(x)}$ spaces were explored  in the papers \cite{EdMe}, \cite{EdKoMe1}, \cite{AshKoMe}, \cite{KMS} (see also the monograph \cite{MeMo}).

Historically, one and two weight Hardy inequalities on the cone of  non-negative decreasing  functions defined on ${\Bbb{R}}_+$ in the classical Lebesgue spaces  were  characterized by M. A. Arino and B. Muckenhoupt  \cite{AM} and E. Sawyer \cite{Saw} respectively.


It should be emphasized that the operator $I_{\alpha}f(x)$ is the weighted truncated potential. The trace inequity for this operator in the classical Lebesgue spaces was established by E. Sawyer \cite{SW} (see also the monograph \cite{EdKoMe}, Ch.6   for related topics).

In general,   the modular inequality
   $$\int\limits_{0}^{1}\Big|\int\limits_{0}^{x} f(t)dt\Big|^{q(x)} v(x)dx\leq c\int\limits_{0}^{1}\big|f(t)\big|^{p(t)}w(t)dt\ \ \ \ \ \  \eqno{(*)}$$
for the Hardy operator is not valid (see \cite{Si} , Corollary 2.3, for details). Namely
the following fact holds: if there exists a positive constant $c$ such that inequality
(*) is true for all
$f\geq 0, \hbox {where}\  q;\  p;\ w\  \hbox{and}\  v$ are non-negative measurable functions,
 then there exists $b \in [0 \   \  1]$ such that
$w(t) > 0$  for almost every    $t < b $;  $v(x) = 0 $ for almost every  $ x > b$, and
$p(t)$ and $q(x)$ take the same constant values almost everywhere for $t \in (0; b)$ and
$x \in (0; b)\cap\{v\neq 0\}.$

To get the main result we use the following pointwise inequities

  $$c_1(Tf)(x)\leq (R_\alpha f)(x)\leq c_2(Tf)(x),$$
$$c_3(Hg)(x)\leq (I_\alpha g)(x)\leq c_4(Hg)(x),$$
for non-negative decreasing functions, where $c_1$, $c_2$, $c_3$ and $c_4$ are constants are independents of $f$, $g$ and $x$, and
 $$Tf(x)=\frac{1}{x}\int\limits_{0}^{x}f(t)dt, \ \ \ \ \ \ \ \ \ \ \ \ \ \ \ \ \
 Hg(x)=\frac{1}{|x|^{n}}\int\limits_{|y|<|x|}g(y)dy.$$

In the sequel by the symbol $Tf\approx Tg$ we means that there are positive constants $c_1$ and $c_2$ such that $c_1 Tf(x)\leq Tg(x)\leq c_2 Tf(x).$ Constants in inequalities will be mainly denoted by $c$ or $C$; the symbol ${\Bbb{R}}_+$ means the interval $(0, +\infty)$.

\section{preliminaries }

 We say that a radial function $f:\mathbb{R}^{n}\longrightarrow \mathbb{R}_{+}$ is decreasing if there is a decreasing  function $g:\mathbb{R}_{+}\longrightarrow \mathbb{R}_{+}$ such that  $g(|x|)=f(x),$ \  $x\in \mathbb{R}^{n}.$ We will denote $g$ again by $f.$
Let $p:\mathbb{R}^{n}\longrightarrow \mathbb{R}^{n}$ be a measurable function, satisfying the conditions
$p^-=\hbox{ess}\!\!\!\!\ \inf\limits_{x\in\mathbb{R}^{n} }p(x)>0,\ p^+=\hbox{ess}\!\!\!\ \sup\limits_{x\in\mathbb{R}^{n} }p(x)<\infty. $

Given $p:\mathbb{R}^{n}\longrightarrow \mathbb{R}_{+}$ such that $0<p^-\leq p^+<\infty$,  and a non-negative measurable function (weight) $u$ in $\mathbb{R}^{n},$ let us define the following local oscillation of $p:$
$$\varphi_{p(\cdot),u(\delta)}=\hbox{ess}\!\!\!\!\!\!\!\!\!\!\!\!\!\!\!\!\ \!\!\ \sup\limits_{x\in B(0,\delta)\cap\ \hbox{\hbox{supp u}}}\!\!\!\ \ p(x)\  -\ \  \ \   \hbox{ess}\!\!\!\!\!\!\!\!\!\!\!\!\!\!\!\!\ \!\!\!\ \inf\limits_{x\in B(0,\delta)\cap\ \hbox{\hbox{supp u}}}\!\!\!\ \ p(x),$$
where $B(0,\delta)$ is the ball with center $0$ and radius $\delta$.

We observe that $\varphi_{p(\cdot),u(\delta)}$ is non-decreasing and positive function such that
$$\lim\limits_{\delta\rightarrow\infty}\varphi_{p(\cdot),u(\delta)}=p^+_u-p^-_u, \eqno{(1)}$$
where $p^+_u$ and $p^-_u$ denote the essential infimum  and supremum  of $p$ on the support of $u.$ respectively.\\

 By the similar manner it is defined (see \cite{SJ}) the function $\psi_{p(\cdot),u(\eta)}$ for an exponent $ p:{\Bbb{R}}_+ \mapsto {\Bbb{R}}_+$ and weight $v$ on ${\Bbb{R}}_+$:
 $$\varphi_{p(\cdot),v(\varepsilon)}=\hbox{ess}\!\!\!\!\!\!\!\!\!\!\!\!\!\!\!\!\ \!\!\ \sup\limits_{x\in B(0,\delta)\cap\ \hbox{\hbox{supp v}}}\!\!\!\ \ p(x)\  -\ \  \ \   \hbox{ess}\!\!\!\!\!\!\!\!\!\!\!\!\!\!\!\!\ \!\!\!\ \inf\limits_{x\in (0, \eta)\cap\ \hbox{\hbox{supp v}}}\!\!\!\ \ p(x),$$

Let $D({\Bbb{R}}_+)$ be the class of non-negative decreasing functions on ${\Bbb{R}}_+$ and let $DR({\Bbb{R}}^n)$  be the class of all non--negative  radially decreasing   functions on ${\mathbb{R}}^{n}$.
Suppose that $u$ is measurable a.e.  positive  function (weight) on ${\mathbb{R}}^{n}$. We denote by $L^{p(x)}(u,{\mathbb{R}}^{n})$,
the class of all non--negative functions on ${\mathbb{R}}^{n}$ for which
$$
S_p(f)=\int\limits_{\mathbb{R}^{n}}|f(x)|^{p(x)}u(x) d\mu(x) < \infty.
$$

For essential properties of  $L^{p(x)}$ spaces we refer to the papers \cite{KR} \cite{Sa} and  the monographs \cite{DHHR}, \ \cite{CF}.

 Under the symbol $L^{p(x)}_{dec}(u,{\mathbb{R}}_{+})$ we mean the class of non-negative decreasing  functions on $\mathbb {R}_{+} $  from  $L^{p(x)}(u,{\mathbb{R}}^{n})\cap DR({\Bbb{R}}^n)$.

 Now we list the well-known results regarding one-weight inequality for the operator $T$. For the following statement we refer to \cite{AM}. \\
     {\bf Theorem A}. {\em Let  $r$ be constant such that $0<r<\infty.$  Then the inequity
$$\int\limits_{0}^{\infty}v(x)(Tf(x))^{r}dx\leq c \int\limits_{0}^{\infty}v(x)(f(x))^{r}dx,\ \ \ \ \  f\in L^{r}(v,\mathbb{R}_{+}),\; f\downarrow \eqno{(2)}$$
for a weight $v$ holds, if and only if there exists a positive constant $C$ such that for all $s>0$
$$\int\limits_{s}^{\infty}\bigg(\frac{s}{x}\bigg)^{r}v(x)dx\leq C\int\limits_{0}^{s}v(x)dx. \eqno{(3)}$$}

\vskip+0.2cm

Condition $(3)$ is called $B_r$ condition and was introduced in \cite{AM}.

\vskip+0.2cm

{\bf Theorem B}\cite{SJ}. {\em Let $v$ be a weight on $(0,\infty)$ and $p:\mathbb{R}_{+}\longrightarrow \mathbb{R}_{+}$
such that $0<p^-\leq p^+<\infty$, and assume that $\psi_{p(\cdot),v(0^+)}=0.$ The following facts are equivalent:\\
 (a) There exists a positive constant $c$ such that for any  $f\in D({\Bbb{R}}_+)$,
 $$\int\limits_{0}^{\infty}\big(Tf(x)\big)^{p(x)}v(x)dx\leq C\int\limits_{0}^{\infty}\big(f(x)\big)^{p(x)}v(x)dx.\eqno{(4)}$$
(b) For any $r, s>0,$
$$\int\limits_{r}^{\infty}\bigg(\frac{r}{sx}\bigg)^{p(x)}v(x)dx\leq C\int\limits_{0}^{r}\frac{v(x)}{s^{p(x)}}dx.\eqno{(5)}$$
(c)  $p_{|_{\mbox{supp v}}}\equiv p_0 $\ \ \  a.e  and $v\in B_{p_{0}}.$}
\\

\begin{proposition}\label{th3.1}
For the operators $T, H, R_\alpha$ and $I_\alpha$,   the following relations hold:\\
$(a)$
$$  R_\alpha f\approx Tf,  \ \ \ \  \ \  \ 0<\alpha<1,\ \ \ \  f \in D({\Bbb{R}}_+) ;$$
$(b)$
$$ I_\alpha g\approx Hg,  \ \ \ \ \ \ \ \ \  0<\alpha<n,\ \ \ \ g\in DR({\Bbb{R}}^n).$$

\end{proposition}
\begin{proof}$(a)$  Upper estimate. Represent $R_\alpha f$ as follows:
$$R_\alpha f(x)=\frac{1}{x^{\alpha}}\int\limits_{0}^{x/2}\frac{f(t)}{(x-t)^{1-\alpha}}dt+
\frac{1}{x^{\alpha}}\int\limits_{x/2}^{x}\frac{f(t)}{(x-t)^{1-\alpha}}dt=S_1(x)+S_2(x).$$
Observe that if $t<x/2,$ then $x/2<x-t$. Hence
$$S_1(x)\leq c\frac{1}{x}\int\limits_{0}^{x/2}f(t)dt\leq cTf(x),$$
where the positive constant $c$ does not depend on $f$ and $x$. Using the fact that $f$ is decreasing  we find that
$$S_2(x)\leq cf(x/2)\leq cTf(x).$$

Lower estimate follows immediately by using the fact that $f$ is non-negative and the obvious estimate $x-t\leq x$ and
$0<t<x.$

$(b)$ Upper estimate. Let us represent the operator $I_\alpha$ as follows:
$$I_\alpha g(x)=\frac{1}{|x|^{\alpha}}\int\limits_{|y|<|x|/2}\frac{g(y)}{|x-y|^{n-\alpha}}dy+
\frac{1}{|x|^{\alpha}}\int\limits_{|x|/2<|y|<|x|}\frac{g(y)}{|x-y|^{n-\alpha}}dy$$
$$\hspace{-5.5cm}=: S'_1(x)+S'_2(x).$$
Since $|x|/2\leq |x-y|$ for $|y|<|x|/2$ we have that
$$S'_1(x)\leq\frac{c}{|x|^{n}}\int\limits_{|y|<|x|/2}g(y)dy\leq cHg(x).$$
 Taking into account the fact that  $f$ is radially decreasing on $\mathbb{R}^{n}$ we find that there is a decreasing function $f:\mathbb{R}_{+}\longrightarrow \mathbb{R}_{+}$ such that
$$S'_2(x)\leq f(|x|/2)\cdot\frac{1}{|x|^{\alpha}}\int\limits_{|x|/2<|y|<|x|}|x-y|^{\alpha-n}dy$$
Let $F_x=\{y:|x|/2<|y|<|x|\}$. Then we have
$$\hspace{-6cm}\int\limits_{F_x}|x-y|^{\alpha-n}dy=\int\limits_{0}^{\infty}\big|\{y\in F_x:|x-y|^{\alpha-n}>t\}\big|dt$$
$$\leq \int\limits_{0}^{|x|^{\alpha-n}}\big|\{y\in F_x:|x-y|^{\alpha-n}>t\}\big|dt+
\int\limits_{|x|^{\alpha-n}}^{\infty}\big|\{y\in F_x:|x-y|^{\alpha-n}>t\}\big|dt$$
$$\hspace{-11cm}=:I_1+I_2.$$
It is easy to see that
$$I_1\leq\int\limits_{0}^{|x|^{\alpha-n}}|B(0,|x|)|dt=c|x|^{\alpha},$$
while using the fact that $\frac{n}{n-\alpha}>1$ we find that
$$I_2\leq\int\limits_{|x|^{\alpha-n}}^{\infty}\big|\{y\in F_x:|x-y|\leq t^\frac{1}{\alpha-n}\}\big|dt
\leq c\int\limits_{|x|^{\alpha-n}}^{\infty}t^\frac{n}{\alpha-n}dt=c_{\alpha,n}|x|^{\alpha}.$$
Finally we conclude  that
$$S'_2(x)\leq cf(|x|/2)\leq cHf(x).$$
Lower estimate follows immediately by using the fact that $f$ is non-negative and the obvious estimate $|x-y|\leq |x|,$ where $0<|y|<|x|.$
\end{proof}

We will also need the following statement:

\begin{lemma}
Let $r$ be a constant such that  $0<r<\infty.$ Then the inequality
$$\int\limits_{\mathbb{R}^{n}}\big(Hf(x)\big)^{r}u(x)dx\leq C\int\limits_{\mathbb{R}^{n}}\big(f(x)\big)^{r}u(x)dx,\ \ \ \ \ \ \  \ \ \  f\in L^{r}_{dec}(u,\mathbb{R}^{n}) \eqno{(6)}$$
holds, if and only if there exists a positive constant C such that for all $s>0,$
$$\int\limits_{|x|>s}\Big(\frac{s}{|x|}\Big)^{r}|x|^{r(1-n)}u(x)dx\leq C\int\limits_{|x|<s}|x|^{r(1-n)}u(x)dx. \eqno{(7)}$$

\end{lemma}
\begin{proof}
We shall see that inequality $(6)$ is equivalent to the inequality
$$\int\limits_{0}^{\infty}\tilde{u}(t)\big(T\bar{f}(t)\big)^{r}dt\leq C\int\limits_{0}^{\infty}\tilde{u}(t)\big(\bar{f}(t)\big)^{r}dt,$$
where $\tilde{u}(t)=t^{(n-1)(1-r)}\bar{u}(t)$, $\bar{f}(t)=t^{n-1}f(t)$ and  $\bar{u}(t)=\int\limits_{S_0}u(t\bar{x})d\sigma(\bar{x}).$\\
Indeed, using polar the coordinates in $\mathbb{R}^{n}$ we have
\begin{eqnarray*}
\int\limits_{\mathbb{R}^{n}}\big(Hf(x)\big)^{r}u(x)dx&=& \int\limits_{\mathbb{R}^{n}}
u(x)\bigg(\frac{1}{|x|^{n}}\int\limits_{|y|<|x|}f(y)dy\bigg)^{r}dx  \\
&=& \int\limits_{0}^{\infty}t^{n-1}\bigg(\frac{1}{|t|^n}\int\limits_{|y|<|x|}f(y)dy\bigg)^{r}
\bigg(\int\limits_{S_0}u(t\bar{x})d\sigma\bar{x}\bigg)dt  \\
&=&C \int\limits_{0}^{\infty}t^{n-1}t^{-nr}t^{r}
\bigg(\frac{1}{t}\int\limits_{0}^{t}\tau^{n-1}f(\tau)d\tau\bigg)^{r}\bar{u}(t)dt \\
&=&C\int\limits_{0}^{\infty}t^{n-1}t^{r(1-n)}\bar{u}(t)\bigg(\frac{1}{t}\int\limits_{0}^{t}\bar{f}(\tau)d\tau\bigg)^{r}dt\\
&\leq& C\int\limits_{0}^{t}\tilde{u}(t)\big(\bar{f}(t)\big)^{r}dt\\
&=& C t^{(n-1)(1-r)}t^{(n-1)r}\big(f(t)\big)^{r}dt\\
&=& C\int\limits_{\mathbb{R}^{n}}\big(f(x)\big)^{r}u(x)dx.
\end{eqnarray*}
\end{proof}

\section{the main results }

To formulate the main results we need to prove

\begin{proposition}\label{th4.2}
Let $u$ be a weight on $\mathbb{R}^{n}$ and $p:\mathbb{R}^{n}\longrightarrow \mathbb{R}_{+}$
such that $0<p^-\leq p^+<\infty$, and assume that $\varphi_{p(\cdot),u(0+)}=0.$ The following statements are equivalent:\\
(a) There exists a positive constant $C$ such that for any $f\in DR({\Bbb{R}}^n)$,
$$\int\limits_{\mathbb{R}^{n}}\big(Hf(x)\big)^{p(x)}u(x)dx\leq C\int\limits_{\mathbb{R}^{n}}\big(f(x)\big)^{p(x)}u(x)dx.\eqno{(8)}$$
(b) For any $r, s>0$,
$$\int\limits_{|x|>r}\bigg(\frac{r}{s|x|^{n}}\bigg)^{p_0}u(x)dx\leq C\int\limits_{B(0,r)}\frac{|x|^{(1-n)p_0}u(x)}{s^{p_0}}dx.\eqno{(9)}$$
(c)  $p_{|_{{\textit{supp u}}}} \equiv p_0 $\ \ \  a.e  and $u\in B_{p_{0}}.$
\end{proposition}
\begin{proof} We use the arguments of \cite{SJ}.
To show  that (a) implies (b) it is enough to test the modular inequality (8)
for the function $f_{r,s}(x)=\frac{1}{s}\chi_{B(0,r)}(x)|x|^{1-n}$, $s, r>0$. Indeed, it can be checked that
$$
Hf_{r,s}(x)=\left\{
  \begin{array}{ll}
    \frac{1}{|x|^{n}s}\int\limits_{|y|\leq|x|}|y|^{1-n}dy, & \hbox{if}\ |x|\leq r;  \vspace{.5cm}\\
     \frac{1}{|x|^{n}s}\int\limits_{|y|\leq r}|y|^{1-n}dy,  & \hbox{if}\ |x|> r
  \end{array}
\right. .
$$

Further, we find that
$$\int\limits_{|x|>r}u(x)\big(Hf_{r,s}\big)^{p(x)}dx
\leq\int\limits_{\mathbb{R}^{n}}u(x)\big(Hf_{r,s}\big)^{p(x)}dx
\leq C\int\limits_{\mathbb{R}^{n}}u(x)\bigg(\frac{1}{s}\chi_{B(0,r)}(x)|x|^{1-n}\bigg)^{p(x)}dx.$$
Therefore
$$\int\limits_{|x|>r}u(x)\bigg(\frac{r}{s|x|^{n}}\bigg)^{p(x)}dx\leq C\int\limits_{B(0,r)}\frac{|x|^{(1-n)p(x)}u(x)}{s^{p(x)}}dx. $$
To obtain $(c)$ from $(b)$ we are going to prove that condition $(b)$ implies that $\varphi_{p(\cdot),u(\delta)}$ is a constant function, namely $\varphi_{p(\cdot),u(\delta)}=p^+_u-p^-_u  \ \ \hbox{for all } \delta>0.$  This fact and the hypothesis on $\varphi_{p(\cdot),u(\delta)}$ implies that $\varphi_{p(\cdot),u(\delta)}\equiv 0,$ and hence due to $(1)$,
$$p_{|_{\mbox{supp u}}}\equiv p^+_u- p^-_u\equiv p_0 \ \ \ \ \hbox{a.e.}.$$  Finally  $(9)$ means that $u\in B_{p_{0}}.$
Let us suppose that $\varphi_{p(\cdot),u}$ is not constant. Then one of the following conditions  hold:\\
(i) there exists $\delta>0$ such that
$$\alpha=\hbox{ess}\!\!\!\!\!\!\!\!\!\!\!\!\!\!\!\!\ \sup\limits_{x\in B(0,\delta)\cap \hbox{supp u} }\!\ p(x)<p^+_u<\infty,\eqno{(10)}$$
and  hence, there exists $\epsilon>0$ such that
$$\big|\{|x|>\delta:p(x)\geq\alpha+\epsilon\}\cap\mbox{supp}\ \mbox{u}\big|>0,$$
or\\
(ii) there exists $\delta>0$ such that
$$\beta=\hbox{ess}\!\!\!\!\!\!\!\!\!\! \!\!\!\!\!\!\!\ \inf\limits_{x\in B(0,\delta)\cap \hbox{supp}\ u}\!\ p(x)>p^-_u>0,\eqno{(11)}$$
and  then, for some  $\epsilon>0,$
$$\big|\{|x|>\delta:p(x)\leq\beta-\epsilon\}\cap\mbox{supp}u\big|>0.$$
In the case (i) we observe that condition (b)  for $r=\delta$, implies that
$$\int\limits_{|x|>\delta}\bigg(\frac{\delta}{s}\bigg)^{p(x)}\frac{u(x)}{|x|^{np(x)}}dx\leq C\int\limits_{B(0,\delta)}\frac{|x|^{(1-n)p(x)}u(x)}{s^{p(x)}}dx.$$
 Then using (10) we obtain, for $s<\min(1,\delta)$,
$$\bigg(\frac{\delta}{s}\bigg)^{\alpha+\epsilon}\int\limits_{\{|x|\geq\delta:p(x)
\geq\alpha+\epsilon\}}\frac{u(x)}{|x|^{np(x)}}dx\leq\frac{C}{s^{\alpha}}\int\limits_{B(0,\delta)}u(x)|x|^{(1-n)p(x)}dx,$$
which is clearly a contradiction if we let $s\downarrow0.$
Similarly in the case (ii) let us consider the same condition (b) for $r=\delta$, and fix now $s>1.$  Taking into account $(11)$ we find that:
$$\frac{1}{s^{\beta-\epsilon}}\int\limits_{\{|x|\geq\delta:p(x)
\leq\beta-\epsilon\}}\bigg(\frac{\delta}{|x|^{n}}\bigg)^{p(x)}u(x)dx
\leq\frac{C}{s^{\beta}}\int\limits_{B(0,\delta)}|x|^{(1-n)p(x)}u(x)dx,$$
which  is a contradiction if we let $s\uparrow\infty.$\\
Finally, the fact that condition (c) implies (a) follows from[1,Theorem 1.7]
\end{proof}

\begin{theorem}
Let $u$ be a weight on $(0,\infty)$ and $p:\mathbb{R}_{+}\longrightarrow \mathbb{R}_{+}$
such that $0<p^-\leq p^+<\infty$. Assume that $\psi_{p(\cdot),v(0^+)}=0.$ The following facts are equivalent:\\
$(i)$ There exists a positive constant $C$ such that for any $f\in D({\Bbb{R}}_+)$,
$$\int\limits_{\mathbb{R}_+}\big(R_{\alpha}f(x)\big)^{p(x)}v(x)dx\leq C\int\limits_{\mathbb{R}_+}\big(f(x)\big)^{p(x)}v(x)dx. $$

$(ii)$ condition $(5)$ holds;

$(iii)$ condition $(c)$ of Theorem $B$ is be satiesfied.
\end{theorem}
\begin{proof}
Proof follows by using Theorems B and Proposition \ref{th3.1}(a).
\end{proof}
\begin{theorem}
Let $u$ be a weight on $\mathbb{R}^{n}$ and $p:\mathbb{R}^{n}\longrightarrow \mathbb{R}_{+}$
such that $0<p^-\leq p^+<\infty$, and assume that $\varphi_{p(\cdot),u(0^+)}=0.$ The following facts are equivalent:\\
$(i)$ There exists a positive constant $C$ such that for any $f\in DR({\Bbb{R}}^n)$,
$$\int\limits_{\mathbb{R}^n}\big(I_{\alpha}f(x)\big)^{p(x)}u(x)dx\leq C\int\limits_{\mathbb{R}^n}\big(f(x)\big)^{p(x)}u(x)dx.$$

$(ii)$ condition $(9)$ holds;

$(iii)$ condition $(c)$ of Proposition \ref{th4.2} holds.
\end{theorem}
\begin{proof}
Proof follows by using Propositions  \ref{th4.2} and Proposition \ref{th3.1} (b).
\end{proof}

\vskip+1cm
{\bf Acknowledgement.} The authors are grafeful to Prof. A. Meskhi for drawing out attention to the problem studied in this paper and helpful remarks.

\vskip+0.2cm

Authors' Addresses:
\vskip+0.1cm

%
%
%
%
%


G. Murtaza:

Department of Mathematics,

 GC University, Faisalabad, Pakistan

Email: gmnizami@@googlemail.com

 M. Sarwar:

 Department of Mathematics,

 University of Malakand, Chakdara,

  Dir Lower, Khyber Pakhtunkhwa, Pakistan

  E-mail: sarwar@uom.edu.pk; sarwarswati@gmail.com

 \end{document}